\documentclass[a4paper]{amsart}
\usepackage[utf8x]{inputenc}
\usepackage{mathrsfs}
\usepackage{graphicx, caption}
\newtheorem{MainTheorem}{Theorem}

\newtheorem{theorem}{Theorem}[section]
\newtheorem{proposition}[theorem]{Proposition}
\theoremstyle{definition}
\newtheorem{remark}[theorem]{Remark}

\newcommand{\evalat}[1]{\bigr\rvert_{#1}}
\newcommand{\set}[2]{\ensuremath{\{#1 \,\colon #2\}}}
\usepackage{enumerate}
\makeatletter
\renewcommand\theenumi{\@alph\c@enumi}
\renewcommand\theenumii{\@alph\c@enumii}
\renewcommand\theenumiii{\@alph\c@enumiii}
\renewcommand\theenumiv{\@alph\c@enumiv}

\def\@map#1#2[#3]{\mbox{$#1 \colon #2 \longrightarrow #3$}}
\def\map#1#2{\@ifnextchar [{\@map{#1}{#2}}{\@map{#1}{#2}[#2]}}
\makeatother
\newcommand{\norm}[1]{\lVert #1 \rVert}
\newcommand{\N}{\ensuremath{\mathbb{N}}}
\newcommand{\Z}{\ensuremath{\mathbb{Z}}}
\newcommand{\R}{\ensuremath{\mathbb{R}}}
\newcommand{\SI}{\ensuremath{\mathbb{S}^1}}
\newcommand{\K}{\ensuremath{\mathcal{K}}}
\newcommand{\Cyl}{\ensuremath{\SI\times\K}}
\renewcommand{\P}{\ensuremath{\mathscr{P}}}
\newcommand{\U}{\ensuremath{\mathscr{U}}}

\newcommand{\Li}{\ensuremath{\mathfrak{L}_1}}
\newcommand{\Lip}{\ensuremath{\mathfrak{L}_1^+}}
\newcommand{\rotn}[1][F]{\ensuremath{\rho_{\scriptscriptstyle #1}}}
\newcommand{\oml}[1][T]{\ensuremath{\omega_{\scriptscriptstyle #1}}}
\newcommand{\rr}[1][\rho(F)]{\ensuremath{\varPhi_{\scriptscriptstyle #1}}}
\newcommand{\rro}{\rr[\rho]}
\DeclareMathOperator{\Cl}{Cl}
\DeclareMathOperator{\Bd}{Bd}
\DeclareMathOperator{\Orb}{Orb}
\newcommand{\Orbn}{\Orb^{\raise0.1ex\hbox{-}}}
\DeclareMathOperator{\Con}{Const}
\DeclareMathOperator{\Rot}{Rot}
\DeclareMathOperator{\Card}{Card}
\DeclareMathOperator{\gr}{graph}
\providecommand{\norminf}[1]{\lVert#1\rVert_\infty}
\title[Uncountable attracting sets in the cylinder]{
Coexistence of uncountably many\\
attracting sets for\\
skew-products on the cylinder}
\author[Ll. Alsed\`a]{Llu\'{\i}s Alsed\`a}
\address{Departament de Matem\`atiques, Edifici Cc, Universitat
Aut\`onoma de Barcelona, 08913 Cerdanyola del Vall\`es, Barcelona,
Spain}
\email{alseda@mat.uab.cat}

\author[S. Costa]{Sara Costa}
\address{Departament de Matem\`atiques, Edifici Cc, Universitat
Aut\`onoma de Barcelona, 08913 Cerdanyola del Vall\`es, Barcelona,
Spain}
\email{sara.costa.romero@gmail.com}
\subjclass[2010]{Primary 37E10, Secondary 37E15.}
\thanks{The authors have been partially supported by the MEC grant
number MTM2008-01486 and MTM2011--26995--C02--01.}
\subjclass{Primary: 37C55, 34D08, 37C70}
\keywords{Quasiperiodically forced system, rotation interval,
attracting set, coexistence of attractors, semiconjugacy}
\makeatletter
\gdef\footnotemark{relax}
\makeatother
\begin{document}
\begin{abstract}
The aim of this paper is to show that the existence of attracting sets
for quasiperiodically forced systems can be extended to appropriate
skew-products on the cylinder, homotopic to the identity, in such a
way that the general system will have (at least) one attracting set
corresponding to every irrational rotation number $\rho$ in the
rotation interval of the base map. This attracting set is a copy of
the attracting set of the system quasiperiodically forced by a (rigid)
rotation of angle $\rho$. This shows the co-existence of uncountably
many attracting sets, one for each irrational in the rotation interval
of the basis.
\end{abstract}
\maketitle
\section{Introduction}
We want to show that the existence of attracting sets for
quasiperiodically forced systems can be extended to a class
of skew-products on the cylinder which are homotopic to
the identity. These systems have an attracting set
corresponding to every irrational rotation number $\rho$ in the
rotation interval of the base map. This attracting set is a copy of
the attracting set of the system quasiperiodically forced by a (rigid)
rotation of angle $\rho$. In particular we show that the systems
from our class can have uncountably many coexisting attracting sets
(one for each irrational in the rotation interval of the base map).

To better explain the above and to state the main result of the paper
we need to recall the basics of rotation theory on the circle, define
what we understand by an attracting set and to fix some notation.

In what follows the circle $\R/\Z$ will be denoted by $\SI$. To
simplify the notation, given $x\in\R$, we will identify $[x] \in \SI$
with its representative in $[0,1)$ (that is, with the fractional part
of $x$, denoted by $\lfloor x \rfloor$).

It is well known that there exists a natural projection
{\map{e}{\R}[\SI]} defined by $e(x):= \lfloor x \rfloor$ and that any
continuous circle map $f$ lifts to a continuous map {\map{F}{\R}}
(called a \emph{lifting of $f$}) in such a way that $f\circ e = e\circ
F$. If $F$ is a lifting of $f$, then $F + n$ is also a lifting of $f$
for every $n \in \Z$ and there exists $d \in \Z$ such that $F(x+1) =
F(x) + d$ for every $x \in \R$. Such integer $d$ is called the
\emph{degree of $f$}.

In this paper we are only interested in continuous degree one circle
maps. These are continuous maps such that $F(x+1) = F(x) + 1$ for
every $x\in\R$ and every lifting $F$. We denote by $\Li$ the set of
all liftings of continuous circle maps of degree one.

For each $F \in \Li$ and $x \in \R$ we define the
\emph{$F$-rotation number of $x$} as
\[
   \rotn(x) := \limsup_{n\to\infty}\frac{F^n(x)-x}{n},
\]
and the \emph{rotation set of $F$} as:
\[
 \Rot(F) := \set{\rotn(x)}{x\in\R}.
\]
From \cite{Ito} we know that $\Rot(F)$ is a closed interval
of $\R$.

To simplify the notation we will denote by $\Rot_I(F)$ the set of
irrationals in $\Rot(F)$.

A circle map $f$ is said to be \emph{non-decreasing} whenever it
has a non-decreasing lifting $F$.

We denote the cylinder by {\Cyl} (where $\K$ is either a
closed interval of $\R$ containing zero, $\R^+$ or $\R$ itself).
We consider the class of skew-products on {\Cyl} of the form:
\begin{equation}\label{SOri}
\begin{pmatrix} \theta_{n+1}\\ x_{n+1} \end{pmatrix} =
T\begin{pmatrix}\theta_{n}\\x_{n} \end{pmatrix}
\qquad\text{where}\qquad
T\begin{pmatrix}\theta\\x \end{pmatrix} =
\begin{pmatrix} f(\theta)\\ p(x)q(\theta)\end{pmatrix},
\end{equation}
$f$ is a continuous circle map of degree one with a lift $F$ such that
$\Rot_I(F)$ is non-empty, $q$ is a continuous map from $\SI$ to $\K$
and $p$ is a continuous map from $\K$ to itself.

The function $p$ plays an essential role in assuring that models of
the above type have attracting sets (in a sense to be made precise
below). Standard assumptions found in the literature on the function
$p$ are, $p(0) = 0$ and, for instance,
\begin{itemize}
\item $p$ is strictly increasing and strictly concave when $\K$
is $\R^+$ (see \cite{Keller,Haro}) or strictly concave in the
positives and strictly convex in the negatives when $\K$ is the
whole $\R$ (see \cite{GOPY,AlCo}). Concrete examples or this map are
$\tanh(x)$ in \cite{Keller,GOPY} and $x^\alpha$ with $0 < \alpha < 1$
in \cite{Haro}.
\item $p$ is unimodal and $\K = [0,1]$ or bimodal and $\K = [-1,1]$
(see \cite{AM,AlCo}).
\end{itemize}
Also, some additional assumptions are demanded on the function $q$
to assure the existence of attracting sets. On the other hand,
in \cite{Keller,GOPY,AlCo,AM} it is required that $q$ vanishes at some
point to assure that the attracting set is pinched, whereas in
\cite{Haro} the function $q$ is assumed to be log-integrable with
respect to an ergodic measure of the basis map.

In what follows we will denote the closure of a set $X$ by $\Cl(X)$.
Also we say that a set $A$ is $f$-invariant whenever $f(A) = A$.

We look for \emph{attracting sets} of System~\eqref{SOri}.
They are defined as follows. Let $\mu$ be an ergodic measure of $f$
and
let $\U$ be a measurable $f$-invariant set such that $\mu(\U) = 1$.
Let {\map{\varphi}{\U}[\K]} be a correspondence whose graph is
$T$-invariant on $\U$ (i.e. $T(\gr(\varphi)) = \gr(\varphi)).$ The
closure of $\gr(\varphi)$ will be called an \emph{attracting set with
support $\U$ and generated by $\varphi$} whenever
\[
 \lim_{n\to\infty} \norm{T^n(\theta, x) - T^n(\theta,z(x))}
 = 0
\]
for every $\theta \in \U$ and $x$ in a subset of $\K$ of positive
Lebesgue measure, and some $z(x) \in \varphi(\theta)$ (in particular,
$\oml(\theta, x) \subset \oml(\theta, \varphi(\theta))$).

For every $\rho\in\R$ we denote the rotation by angle $\rho$ by
$\rro(\theta) := \theta + \rho \pmod{1}$.

The main result of the paper is the following theorem which shows
that the attracting sets of System~\eqref{SOri} can be obtained
from the attracting sets of quasiperiodically forced systems which can
be considered subsystems of System~\eqref{SOri}.

\begin{MainTheorem}\label{t-no_num}
Consider a system of the form \eqref{SOri}. To every
$\rho \in \Rot_I(F)$ we can associate
a measurable $f$-invariant set $\U_\rho \subset \SI,$
a continuous non-decreasing circle map of degree one $h_\rho,$ and
an $f$-ergodic measure $\mu_\rho$ such that
$\mu_\rho(\U_\rho)=1,$
$\rotn(e^{-1}(\Cl(\U_\rho))) = \rho,$
{\map{h_\rho\evalat{\U_\rho}}{\U_\rho}[h_\rho(\U_\rho)]}
is a homeomorphism and $h_\rho(\U_\rho)$ is a dense
$\rr[\rho]$-invariant set.
Additionally, the sets $\Cl(\U_\rho)$ are pairwise disjoint.

Assume that, for every $\rho \in \Rot_I(F),$ the system
\begin{equation}\label{SRot}
\begin{pmatrix} \theta_{n+1}\\ x_{n+1} \end{pmatrix} =
S_\rho\begin{pmatrix}\theta_{n}\\x_{n} \end{pmatrix}
\qquad\text{where}\qquad
S_\rho\begin{pmatrix}\theta\\x \end{pmatrix} =
\begin{pmatrix}
  \rro(\theta)\\
  p(x)q(h_\rho^{-1}(\theta))
\end{pmatrix},
\end{equation}
has an attracting set with support $h_\rho(\U_\rho)$ which is the
closure of the graph of a correspondence
{\map{\varphi_\rho}{h_\rho(\U_\rho)}[\K]}.
Then, the closure of the graph of $\varphi_\rho \circ
h_\rho$ is an attracting set of $T$ with support $\U_\rho.$
Thus, whenever $\Rot(F)$ is non-degenerate,
$T$ has uncountably many attracting sets coexisting dynamically.
\end{MainTheorem}

\begin{remark}
The map $q \circ h_\rho^{-1}$ is continuous in $h_\rho(\U_\rho)$ which
is dense in $\SI$. Hence, if $q \circ h_\rho^{-1}$ is discontinuous,
it has only jump discontinuities in the complement of
$h_\rho(\U_\rho).$
\end{remark}

Next we derive some consequences from Theorem~\ref{t-no_num}
organized in a sequence of remarks.

\begin{remark}[On the invariant measures of
$T$]\label{r-liftedmeasure}
The measure $\mu_\rho$ lifted to the closure of the graph of
$\varphi_\rho \circ h_\rho$ is an invariant measure of $T$. Moreover,
this measure is ergodic if and only if the closure of the graph of
$\varphi_\rho \circ h_\rho$ is a minimal set of $T$. In particular
this is a criterion to decide the undecomposability of the attracting
set into several smaller attracting sets.
\end{remark}

\begin{remark}[On the number of pieces of an
attracting set]\label{r-attsets}
Theorem~\ref{t-no_num} tells us that the attracting sets of systems of
the form \eqref{SOri} are graphs of correspondences from $f$-invariant
subsets $\U \subset\SI$ to the fibres.
As it has been shown in \cite{AlCo} for the case when $f$ is an
irrational rotation and
\begin{itemize}
\item $\K = \R$, $p$ is strictly increasing in $\R$ and strictly
concave in $\R^+$ and strictly convex in $\R^-$; or
\item $p$ is bimodal and $\K = [-1,1],$
\end{itemize}
there are two possibilities for such an attracting set: either its
closure is a minimal attractor or it splits into two different minimal
attractors and each of these attractors is the closure of the graph of
a map from $\U$ to the fibres. This dichotomy is inherited by the
attracting sets of the systems of the form \eqref{SOri} with $f$ an
arbitrary continuous function of degree one that are obtained by
transporting the attracting sets obtained in the case of rotations as
stated in Theorem~\ref{t-no_num}.

Observe that when $q$ is negative, the orbits keep alternating between
$\R^+$ and $\R^-$. Thus, typically, there will be an attracting set
which is a 2-periodic orbit of function graphs.
When $q$ is positive, System~\eqref{SOri} can be
split into two (one restricted to $\R^+$ and the other one restricted
to $\R^-$). Consequently we get two attracting sets intersecting at $x
\equiv 0$.

A particular example of the first case (that is, when the attracting
set is a 2-periodic orbit of function graphs --- $q$ negative) is the
following (see Figure~\ref{fig:ak-discont-neg}): consider
System~\eqref{SOri} with
$f = \rr[\tfrac{\sqrt{5}-1}{2}],$
$\K = \R,$
\[
p(x) = \begin{cases}
 \tanh(x)                                    & \text{if $x\ge 0$}\\
 \tfrac{\tanh(x-2) + \tanh(2)}{1-\tanh(2)^2} & \text{if $x\le 0$}
\end{cases}
\ \text{and}\
q(\theta)=\begin{cases}
       7(\cos(2\pi\theta)-2) & \text{if $\theta \in
[0,\tfrac{1}{2})$}\\
       5(\cos(2\pi\theta)-4) & \text{otherwise}.
      \end{cases}
\]

Also, an example of the case when the attracting set has two minimal
components intersecting at $x \equiv 0$ ($q$ positive) can be
obtained from System~\eqref{SOri} by taking
$f = \rr[\tfrac{\sqrt{5}-1}{2}],$
$\K = [-1,1],$
\[
p(x) = \begin{cases}
        x(1-x)                 & \text{if $x\geq 0$}\\
        \frac{1}{2}x(x+1)(x+2) & \text{if $x\le 0$}
       \end{cases}
\ \text{and}\
q(\theta)=\begin{cases}
       2.1 |\cos(2\pi\theta)| & \text{if $\theta \in
[0,\tfrac{1}{2})$}\\
       2.5 |\sin(2\pi\theta)| & \text{otherwise}.
      \end{cases}
\]
In Figure~\ref{fig:lm-discont} we show the two attracting sets for
this system (one in red and the other one in blue).
\begin{figure}\centering
\begin{minipage}[t]{0.42\textwidth}
\captionsetup{width=\textwidth}
\includegraphics[width=0.65\textwidth,angle=270]{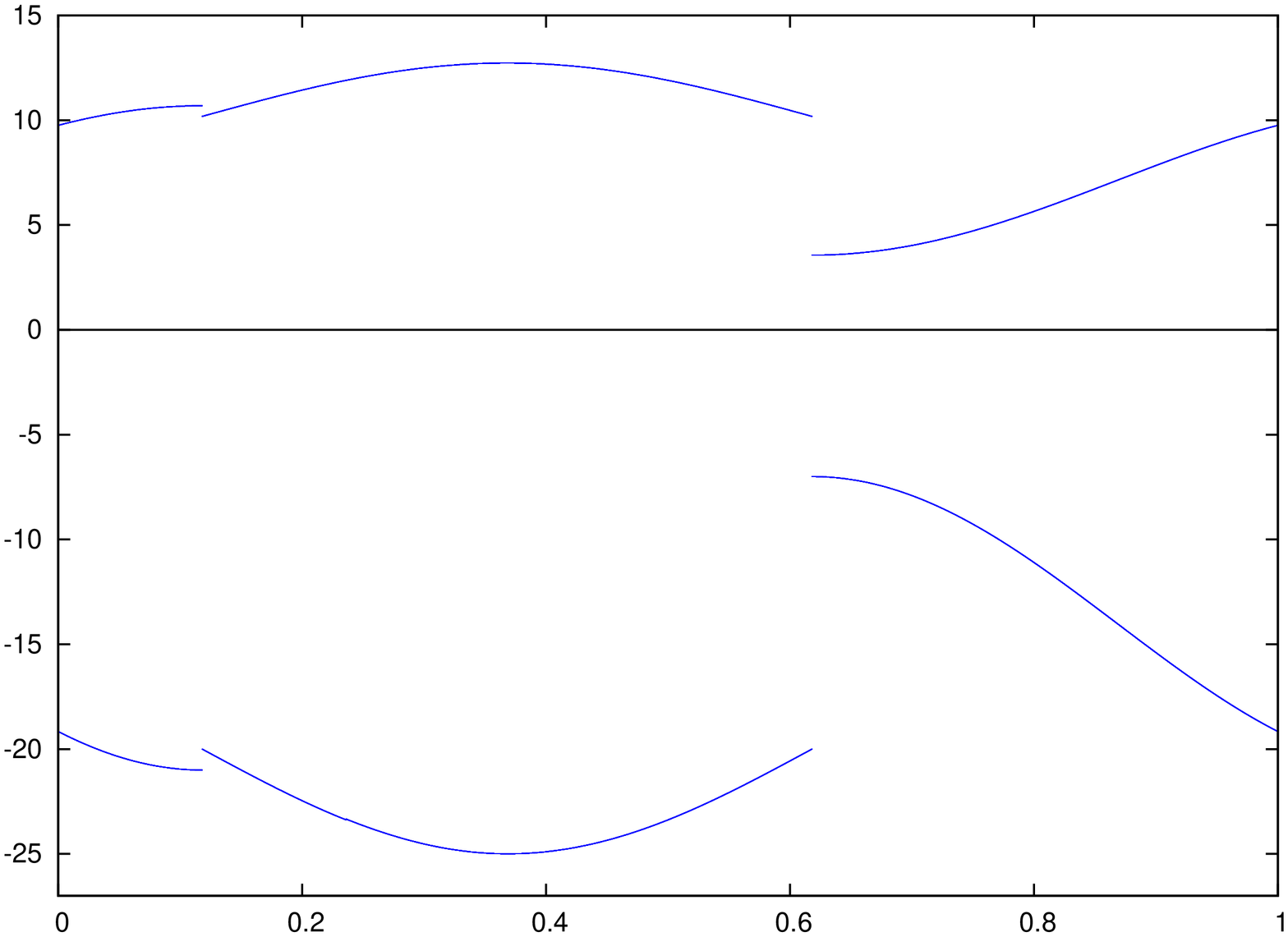}
\caption{The unique attracting set for System~\eqref{SOri}
when $p$ is monotone and $q$ is negative.}\label{fig:ak-discont-neg}
\end{minipage}\qquad
\begin{minipage}[t]{0.42\textwidth}
\captionsetup{width=\textwidth}
\includegraphics[width=0.65\textwidth,angle=270]{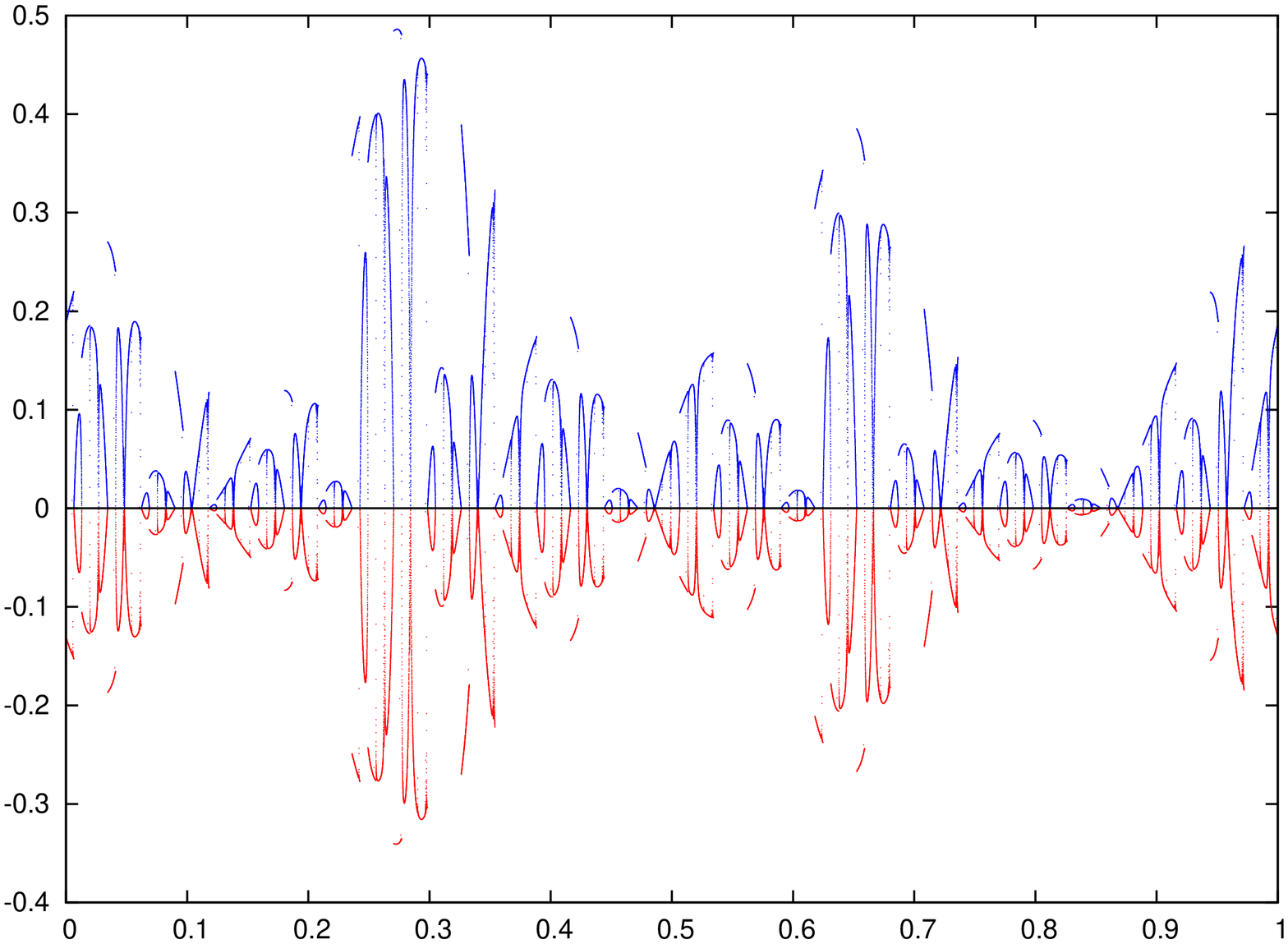}
\caption{The two attracting sets (one in red and the other one in
blue) for System~\eqref{SOri} when $p$ is bimodal and $q$ is
positive.}\label{fig:lm-discont}
\end{minipage}
\end{figure}
\end{remark}

\begin{remark}[On the strangeness of the attracting sets]
An attracting set with support $U$ and generated by $\varphi$ will be
called \emph{strange} if the map $\varphi\evalat{U}$ is discontinuous
in a dense subset of $U$.

Assume that the attracting set with support $h_\rho(\U_\rho)$
generated by $\varphi_\rho$ is strange. Then, there exists a set
$D_\rho,$ dense in $h_\rho(\U_\rho),$ such that $\varphi_\rho$ is
discontinuous at every point of $D_\rho$. Since
$h_\rho\evalat{\U_\rho}$ is a homeomorphism,
$\varphi_\rho \circ h_\rho$
is discontinuous in
$\left(h_\rho\evalat{\U_\rho}\right)^{-1}(D_\rho)$
and this set is dense in $\U_\rho.$ Therefore the attracting set
generated by $\varphi_\rho \circ h_\rho$ is also strange.
\end{remark}

\begin{remark}[On the Lyapunov exponents]
In general the concrete values of the vertical Lyapunov exponents of
Systems~\eqref{SOri} and \eqref{SRot} on the attracting sets, when
they exist, do not have any relation except for the fact that the
attractiveness of the attracting sets implies that both of them are
non-positive.
\end{remark}

Examples of skew-products which satisfy the hypothesis required for
System~\eqref{SRot} can be found in \cite{AlCo}. They are extensions
of known results from \cite{Keller,Haro,AM} to the case where the
function $q$ may have jump discontinuities in an
invariant zero-Lebesgue measured subset of $\SI$.

To prove Theorem~\ref{t-no_num} we need some more knowledge of
rotation theory on the circle and its relation with water functions
and to study the dynamics of non-decreasing degree one circle maps
without periodic points. This will be done in the next section.
Finally, in Section~\ref{proof_main_theorem} we will prove
Theorem~\ref{t-no_num}.

\section{Water functions and non-decreasing degree one
circle maps}\label{WF-AK}

We will start this section with a survey on rotation theory in the
circle and water functions (see \cite[Section~3.7]{ALM}).

We begin by introducing some notation. In what follows $\Lip$ will
denote the class of non-decreasing maps of $\Li$.

If $f$ is a continuous circle map then $\Con(f)$ denotes the
open set of all points $\theta \in \SI$ for which there exists a
neighbourhood $U \ni \theta$ such that $f\evalat{U}$ is constant.
Analogously, if $F$ is a lifting of $f$, then $\Con(F)$ denotes the
open set of all points $x \in \R$ for which there exists a
neighbourhood $U \ni x$ such that $F\evalat{U}$ is constant.
Observe that $\Con(F) = e^{-1}(\Con(f))$.

For every $F \in \Li$ there exists a circle map $F^e$ defined by
\[
F^e:= e \circ F \circ (e\evalat{[0,1)})^{-1}
\]
(observe that $e\evalat{[0,1)}$ is a homeomorphism). Clearly, $F^e$
is continuous and has $F$ as a lifting. The map $F^e$ is called the
\emph{projection of $F$ to $\SI$}. Observe that whenever $F$ is a
lifting of $f,$ then $F^e = f$.

Given $F \in \Li$ we define the \emph{lower} and \emph{upper} liftings
as follows:
\begin{align*}
F_l(x) & := \inf\set{F(y)}{y\ge x}, \text{and}\\
F_u(x) & :=\sup\set{F(y)}{y\le x}.
\end{align*}
Clearly, $F_l$ and $F_u$ belong to $\Lip$, $F_l \le F \le F_u$ and, if
$F$ is non-decreasing then $F = F_l = F_u$. Moreover, if $F,G \in \Li$
are such that $F \le G,$ then $F_l \le G_l$ and $F_u \le G_u$.

Now, given $F \in \Li$ and $0 \le \alpha \le \norminf{F -F_l}$, we
define the \emph{water function of level $\alpha$} as
\[
F_\alpha := (\min\{F,F_l+\alpha\})_u
\]
(observe that, from all said above, $\min\{F,F_l+\alpha\} \in \Li$
and, hence, $F_\alpha$ is well defined). The projection to $\SI$ of
the lifting $F_\alpha$, $F_\alpha^e$, will be denoted by $f_\alpha$.

From the definition and the properties of the upper and lower liftings
it follows that $F_\alpha \in \Lip$ for every $\alpha$, $F_0 = F_l$,
$F_{\scriptscriptstyle \norminf{F-F_l}} = F_u$ and $F_\alpha \le
F_{\alpha'}$ whenever $\alpha \le \alpha'$. Moreover, for every
$\alpha$, $F_\alpha$ coincides with $F$ in the complement of
$\Con(F_\alpha).$ then, since $e(\Con(F_\alpha)) = \Con(f_\alpha)$ it
follows that $f_\alpha$ coincides with $f$ in the complement of
$\Con(f_\alpha)$.

If $F \in \Lip$, then
\[
\rotn(x) = \lim_{n\to\infty}\frac{F^n(x)-x}{n}
\]
and it is independent on $x\in\R$. This number is denoted
as $\rho(F)$ and called the \emph{rotation number of $F$}.
It is well known that the map $F \longmapsto \rho(F)$
from $\Lip$ to $\R$ is continuous and non-decreasing.

On the other hand, for every $F \in \Li$,
$\Rot(F) = [\rho(F_l), \rho(F_u)]$.

From all said above it follows that the map
$\alpha \longmapsto \rho(F_\alpha)$
from
$[0, \norminf{F-F_l}]$ to $\Rot(F)$
is continuous, onto and non-decreasing.

Now we dynamically study the maps from $\Lip$ with irrational rotation
number. One of the important facts that we will use in this study is
the semiconjugacy of circle maps with degenerate rotation interval to
a rotation. Let $f$ be a continuous circle map with lifting $F\in\Lip$
such that $\rho(F)$ is irrational. From \cite{Denjoy} we know
that $f$ is semiconjugate to the irrational rotation $\rr$ by a
\emph{non-decreasing} map $h$: $h \circ f = \rr \circ h$.

The result we look for is the following.

\begin{proposition}\label{set-U}
Let $f$ be a continuous circle map with lifting $F\in\Lip$ such that
$\rho(F)$ is irrational. Then $f$ has a measurable invariant set $\U
\subset \SI$ and a unique ergodic measure $\mu$ such that
$\mu(\U)=1,$
$\Cl(\U)$ is disjoint from $\Con(f),$
$h\evalat{\U}$ is a homeomorphism and
$h(\U)$ is a dense $\rr$-invariant set.
If $f$ is not a homeomorphism, then $\U$ is nowhere dense in $\SI$.
\end{proposition}

To prove Proposition~\ref{set-U} we will use the results by Auslander
and Katznelson \cite{AK} on continuous circle maps without periodic
points. To relate both situations observe that if $f$ is a continuous
circle map of degree one with lifting $F \in \Lip$ having irrational
rotation number then, from \cite[Lemma~3.7.2]{ALM} it follows that $f$
has no periodic points. On the other hand, if $f$ has no periodic
points then, from \cite[Sections~3.4--6 and Lemma~3.7.2]{ALM} it
follows that $f$ must have degree one and, if $F$ is a lifting of $f$,
then $\Rot(F)$ is degenerate to an irrational.

Given a circle map $f$ we will denote the \emph{forward} and
\emph{backward} orbit of a point $\theta \in \SI$ by
$\Orb_f(\theta)$ and $\Orbn_f(\theta),$ respectively:
\begin{align*}
\Orb_f(\theta) &:= \set{f^n(\theta)}{n \in \Z^+},\\
\Orbn_f(\theta) &:= \bigcup_{n \in \Z^+} f^{-n}(\theta).
\end{align*}

The the next theorem is a version of \cite[Theorem~2]{AK}. The
unique ergodicity of $f$ is proved at the very end of \cite{AK}. The
fact that $\P$ is uncountable follows from the arguments from
the third paragraph of \cite[Page~380]{AK}.

\begin{theorem}\label{AK-theo}
Let $f$ be a continuous circle map without periodic points. Then $f$
has a unique minimal set $\P$ and $\oml[f](\theta)=\P$ for all
$\theta\in\SI$. Moreover, $f$ has a unique ergodic measure (and,
consequently, its support is precisely $\P$).
The set $\P$  is uncountable and, if $f$ is not a
homeomorphism, it is nowhere dense in $\SI$.
Additionally,
for any $\theta \in \P$, $1 \le \Card(f^{-1}(\theta)\cap \P) \le 2$
and the set of $\theta\in\P$ such that
$\Card(f^{-1}(\theta)\cap \P) = 2$ is countable.
Moreover, for every $\theta \in \P$, there is at most one
$\theta'\in \Orb_f(\theta) \cup\left(\Orbn_f(\theta) \cap \P\right)$
such that
$\Card(f^{-1}(\theta') \cap \P) = 2$.
\end{theorem}

\begin{remark}\label{AK-theo_noconst}
The set $\P$ is disjoint from $\Con(f)$.
Otherwise, there exists a point $\theta \in U \cap \P$ where
$U$ is a connected component of $\Con(f)$. Also, by the minimality of
$\P$, there exists $n \in \N$ such that $f^n(\theta) \in U$. Thus,
since $f\evalat{U}$ is constant, $f(\theta) = f^{n+1}(\theta)$ is
periodic; a contradiction.
\end{remark}

\begin{remark}\label{m_density}
For every $\theta \in \P$ we have $\mu(\theta) = 0,$ where $\mu$ is
the ergodic measure from Theorem~\ref{AK-theo}.
To show it we
assume that there exists $\theta \in \P$ such that $\mu(\theta) > 0$
and set
$A_0 = \{\theta\}$ and $A_i = f^{-1}(A_{i-1}) \cap \P$
for $i=1,2,\ldots.$
Since $\mu$ is $f$-invariant and has support $\P$, it follows by
induction that $\mu(A_i) = \mu(\{\theta\}) > 0$ for every
$i=1,2,\ldots.$
Suppose now that there exists $\nu \in \SI$ and $l > n \ge 0$ such
that $\nu \in A_l \cap A_n$. Then,
\[
 \theta = f^l(\nu) = f^{l-n}(f^n(\nu)) = f^{l-n}(\theta)
\]
and $\theta$ is periodic; a contradiction. This shows that
the sets $A_i$ are pairwise disjoint. Hence,
$\mu \left(\bigcup_{n\ge 0} A_i \right) = \sum_{n=0}^\infty \mu(A_i)$
is infinite, which contradicts the fact that $\mu$ is a finite
measure.
\end{remark}

Now we are ready to prove Proposition~\ref{set-U}.

In what follows we will denote the boundary (that is, the endpoints
when $X$ is homeomorphic to an interval of the real line) of a set $X$
by $\Bd(X).$

\begin{proof}[Proof of Proposition~\ref{set-U}]
As we have said before, $f$ has no periodic points by
\cite[Lemma~3.7.2]{ALM}. So, we can use Theorem~\ref{AK-theo}. We
get that $f$ has a unique ergodic measure and the support of this
measure is $\P$.

Next we construct the set $\U$ by removing some undesirable orbits
from $\P$. To do it we denote by $\mathscr{S}$ the family of all
connected components of $\SI\setminus\P$ and we set
\begin{align*}
\widetilde{\P} &:= \bigcup_{C \in \mathscr{S}} \Bd(C),\\
\U^{c}         &:=
      \left( \bigcup_{\theta \in \widetilde{\P}}
         \left(\bigcup_{n=0}^\infty \Orbn_f(f^n(\theta))\right)
      \right) \cap \P \text{ and} \\
\U & := \P \setminus\U^{c}.
\end{align*}
From the above definitions it follows that
$f(\U^{c}) \subset \U^{c}$ and
$f^{-1}(\theta)\cap \P \subset \U^{c}$
for every $\theta \in \U^{c}.$
Consequently, $f(\U) = \U$ because $f(\P) = \P.$

Since $\P$ is closed the set $\widetilde{\P}$ is at most countable
and, hence, $\U^{c}$ is also at most countable by
Theorem~\ref{AK-theo}. Therefore, $\U^{c}$ is measurable and, by
Remark~\ref{m_density}, $\mu(\U^{c}) = 0.$ Also, the set $\P$ consists
of uncountably many orbits of $f.$
Since $\P$ is measurable and $\mu(\P) = 1$ it follows that
$\U = \P \setminus\U^{c}$ is measurable and $\mu(\U) = 1.$

Notice that, since $\P$ is minimal,
$\P = \Cl(\Orb(\theta)) \subset \Cl(\U) \subset \P$
for every $\theta \in \U$. Hence, $\Cl(\U) = \P.$

The fact that $\Cl(\U)$ is disjoint from $\Con(f)$
follows from Remark~\ref{AK-theo_noconst}. Also, if $f$ is not a
homeomorphism, then $\Cl(\U)$ is nowhere dense in $\SI$ by
Theorem~\ref{AK-theo}.

Since $f$ and $\rr$ are semiconjugate by $h$,
$\rr(h(\U)) = h(f(\U)) = h(\U).$
Consequently, since $\rho(F)$ is irrational,
$\SI = \Cl\left(\Orb_{\rr}(\theta)\right) \subset \Cl(h(\U))$
for every $\theta\in h(\U)$ and, hence, $h(\U)$ is dense in $\SI$.

Now we will show that {\map{h\evalat{\U}}{\U}[h(\U)]} is a
homeomorphism. By the continuity of $h$ we only have to prove that
$h\evalat{\U}$ is one-to-one. Otherwise, there exists an open interval
$J \subset \SI$ such that $\Bd(J) \subset \U$ and $\Card(h(\Bd(J)) =
1$. Since $h$ is non-decreasing, $h\evalat{\Cl(J)}$ is constant. If
$\P \cap J = \emptyset,$ then $J \in \mathscr{S}$ and so $\Bd(J)
\subset \widetilde{\P} \subset \SI\setminus\U$. Thus, there exists a
point $\theta \in \P \cap J.$ By the minimality of $\P$, there exists
$n \in \N$ such that $f^n(\theta) \in J$. Therefore,
\[
h(\theta) = h(f^n(\theta)) = \rr^n(h(\theta));
\]
a contradiction with the fact that $\rho(F)$ is irrational.
\end{proof}

\begin{remark}
It is not difficult to show that in the situation of the above
proposition $f^{-1}(\U) = \U$ and ${\map{f^{-1}\evalat{\U}}{\U}}$ is
a homeomorphism.
\end{remark}

\section{Proof of Theorem~\ref{t-no_num}}\label{proof_main_theorem}

Since
$\alpha \longmapsto \rho(F_\alpha)$
is a continuous map from
$[0, \norminf{F-F_l}]$ onto $\Rot(F),$
to every $\rho \in \Rot_I(F)$ we can associate an
$\alpha_\rho \in [0,\norminf{F-F_l}]$ such that
$\rho(F_{\alpha_\rho}) = \rho.$

Then, for every $\rho \in \Rot_I(F)$, we denote respectively by
$\U_\rho,$ $h_\rho$ and $\mu_\rho$ the set $\U,$ the map $h$ and the
measure $\mu$ given by Proposition~\ref{set-U} for the map
$f_{\alpha_\rho}$ with lifting $F_{\alpha_\rho}$.
Then we have that $\U_\rho$ is measurable and
$f_{\alpha_\rho}$-invariant,
$\mu_\rho(\U_\rho) = 1,$
{\map{h_\rho\evalat{\U_\rho}}{\U_\rho}[h_\rho(\U_\rho)]}
is a homeomorphism and $h_\rho(\U_\rho)$ is a dense
$\rr[\rho]$-invariant set.

Since $f_{\alpha_\rho}$ coincides with $f$ in the complement of
$\Con(f_{\alpha_\rho})$ and $\Cl(\U_\rho)$ is disjoint from
$\Con(f_{\alpha_\rho}),$
\begin{equation}\label{coincidWF}
f\evalat{\Cl(\U_\rho)} =
f_{\alpha_\rho}\evalat{\Cl(\U_\rho)}.
\end{equation}
Hence, $f(\U_\rho) = \U_\rho$ because $\U_\rho$ is
$f_{\alpha_\rho}$-invariant.
Also, since $F_{\alpha_\rho}$ is non-decreasing,
\[
\rotn(e^{-1}(\Cl(\U_\rho))) =
\rotn[F_{\alpha_\rho}](e^{-1}(\Cl(\U_\rho))) =
\rho(F_{\alpha_\rho}) = \rho.
\]
Consequently, if $\rho,\rho' \in \Rot_I(F)$ and $\rho\ne\rho',$ then
$\Cl(\U_\rho) \cap \Cl(\U_{\rho'}) = \emptyset$
since otherwise, for every
$\theta \in e^{-1}(\Cl(\U_\rho) \cap \Cl(\U_{\rho'})),$
$\rho = \rotn(\theta) = \rho';$ a contradiction.

Now we will prove that $\mu_\rho$ is an ergodic measure of $f$.
First we will show that $\mu_\rho$ is $f$-invariant.
From \eqref{coincidWF} we get
\begin{equation}\label{bothmapspreimages}
f^{-1}(A) \cap \U = f^{-1}_{\alpha_\rho}(A) \cap \U
\end{equation}
for every $A \subset \SI.$
Hence, since $\mu_\rho(\U) = 1,$
\[
\mu_\rho(f^{-1}_{\alpha_\rho}(A)) =
\mu_\rho(f^{-1}_{\alpha_\rho}(A) \cap \U) =
\mu_\rho(f^{-1}(A) \cap \U) =
\mu_\rho(f^{-1}(A))
\]
for every measurable set $A \subset \SI.$
Consequently, $\mu_\rho$ is $f$-invariant because it is
$f_{\alpha_\rho}$-invariant.

To prove the $f$-ergodicity of $\mu_\rho$ we will show that
$\mu_\rho(A) \in \{0,1\}$
for every measurable set $A \subset \SI$ such that
$\mu_\rho(f^{-1}(A) \bigtriangleup A)=0$
(where $\bigtriangleup$ denotes the symmetric difference).
By \eqref{bothmapspreimages} we have
\begin{align*}
(f^{-1}(A) \bigtriangleup A) \cap\U
&= (f^{-1}(A)\cap\U) \bigtriangleup (A\cap\U)
 = (f^{-1}_{\alpha_\rho}(A)\cap\U) \bigtriangleup (A\cap\U)\\
&=(f^{-1}_{\alpha_\rho}(A) \bigtriangleup A) \cap\U.
\end{align*}
Consequently,
\[
\mu_\rho(f^{-1}_{\alpha_\rho}(A) \bigtriangleup A) =
\mu_\rho(f^{-1}(A) \bigtriangleup A) = 0
\]
and, since $\mu_\rho$ is $f_{\alpha_\rho}$-ergodic,
$\mu_\rho(A) \in \{0,1\}.$

Now we assume that, for every $\rho \in \Rot_I(F),$ the system
\begin{equation}
\begin{pmatrix} \theta_{n+1}\\ x_{n+1} \end{pmatrix} =
S_\rho\begin{pmatrix}\theta_{n}\\x_{n} \end{pmatrix}
\qquad\text{where}\qquad
S_\rho\begin{pmatrix}\theta\\x \end{pmatrix} =
\begin{pmatrix}
  \rro(\theta)\\
  p(x)q(h_\rho^{-1}(\theta))
\end{pmatrix},
\end{equation}
has an attracting set with support $h_\rho(\U_\rho)$ which is the
closure of the graph of a correspondence
{\map{\varphi_\rho}{h_\rho(\U_\rho)}[\K]}.

We will prove that the closure of the graph of the map
{\map{\widetilde{\varphi}_\rho}{\U_\rho}[\K]}, where
$\widetilde{\varphi}_\rho := \varphi_\rho \circ h_\rho,$ is an
attracting set of $T$ with support $\U_\rho.$ To this end we will
use the map
$H:=(h_\rho, \mathrm{Id})$
which is a homeomorphism
from $\U_\rho \times \K$ to $h_\rho(\U_\rho) \times \K$.

Recall that $h_\rho$ is a semiconjugacy between
$f_{\alpha_\rho}$ and $\rro.$ Thus,
\[
h_\rho \circ f_{\alpha_\rho} = \rro \circ h_\rho
\]
and hence, from \eqref{coincidWF},
\begin{equation}\label{semiconjugacy}
h_\rho(f(\theta))=
h_\rho(f_{\alpha_\rho}(\theta)) =
\rr(h_\rho(\theta))
\end{equation}
for every $\theta \in \U_\rho.$ Therefore,
\begin{align*}
H(T(\theta, z))
 &= \left(h_\rho(f(\theta)), p(z)q(\theta)\right) =
\left(\rro(h_\rho(\theta)),p(z)q(h^{-1}_\rho(h_\rho(\theta)))\right)
\\
 &= S_\rho(H(\theta, z)),
\end{align*}
for every $\theta \in \U_\rho$ and $z\in \K.$
Thus, for each $n \ge 0$ and $z \in \K$,
\begin{equation}\label{conjugate}
\pi_x\left(S_\rho^n(h_\rho(\theta), x)\right) =
\pi_x\left(S_\rho^n(H(\theta,x))\right) =
\pi_x\left(H(T^n(\theta,x))\right) =
\pi_x\left(T^n(\theta,x)\right)
\end{equation}
where {\map{\pi_x}{\Cyl}[\K]} denotes the projection with
respect to the second component.

First we will show that the graph of $\widetilde{\varphi}_\rho$
is $T$-invariant. That is, by \eqref{semiconjugacy}, we have to
prove
\[
\pi_x\left(T(\theta,x)\right) \in
\widetilde{\varphi}_\rho(f(\theta)) =
\varphi_\rho(\rr(h_\rho(\theta)))
\]
for every
$\theta \in \U_\rho$ and
$x \in \widetilde{\varphi}_\rho(\theta).$
By assumption, the graph of
$\varphi_\rho$ is $S_\rho$-invariant:
$
\pi_x\left(S_\rho(h_\rho(\theta), x)\right) \in
\varphi_\rho(\rr(h_\rho(\theta)))
$
for every
$\theta \in \U_\rho$ and
$x \in \varphi_\rho(h_\rho(\theta)) =
\widetilde{\varphi}_\rho(\theta).$
From \eqref{conjugate}, we get
$
\pi_x\left(S_\rho(h_\rho(\theta), x)\right)
= \pi_x\left(T(\theta,x)\right),
$
which proves the $T$-invariance of the graph of
$\widetilde{\varphi}_\rho.$

Now we have to prove that the closure of the graph of
$\widetilde{\varphi}_\rho$ is an attracting set of $T$ with support
$\U_\rho.$
Observe that, for any skew product $R$ on the
cylinder $\Cyl$ and any $n \ge 0,$ $\theta \in \SI$ and $x,z \in \K,$
\[
 \norm{R^n(\theta, x) - R^n(\theta,z)} =
\lvert
     \pi_x\left(R^n(\theta,x)\right) -
     \pi_x\left(R^n(\theta,z)\right)
\rvert.
\]
Therefore, we have to show that
\begin{equation}\label{toprove}
 \lim_{n\to\infty} \left(
   \pi_x\left(T^n(\theta,x)\right) -
   \pi_x\left(T^n(\theta,z(x))\right)
 \right) = 0
\end{equation}
for every $\theta \in \U_\rho$ and $x$ in a subset of $\K$ of positive
Lebesgue measure, and some
$z(x) \in \widetilde{\varphi}_\rho(\theta).$
By assumption we know that
\[
 \lim_{n\to\infty} \left(
   \pi_x\left(S_\rho^n(h_\rho(\theta), x)\right) -
   \pi_x\left(S_\rho^n(h_\rho(\theta), z(x))\right)
 \right) = 0
\]
for every $\theta \in \U_\rho$ and $x$ in a subset of $\K$ of positive
Lebesgue measure, and some
$
z(x) \in
\varphi_\rho\left(h_\rho(\theta)\right) =
\widetilde{\varphi}_\rho(\theta).
$
From \eqref{conjugate} this is equivalent to \eqref{toprove}.
This ends the proof of the theorem.



\begin{thebibliography}{9}
\bibitem{AlCo}
Ll. Alsed{\`a} and S.~Costa
\newblock On the existence of SNA for quasiperiodically forced
skew-products on $\SI \times \R$
\newblock preprint, 2010.

\bibitem{ALM}
Ll. Alsed{\`a}, J.~Llibre, and M.~Misiurewicz.
\newblock {\em Combinatorial dynamics and entropy in dimension one},
volume~5 of {\em Advanced Series in Nonlinear Dynamics}.
\newblock World Scientific Publishing Co. Inc., River Edge, NJ, second
edition,  2000.

\bibitem{AM}
Ll. Alsed{\'a} and M.~Misiurewicz.
\newblock Attractors for unimodal quasiperiodically forced maps.
\newblock {\em J. Difference Equ. Appl.}, 14(10):1175--1196, 2008.

\bibitem{AK}
J.~Auslander and Y.~Katznelson.
\newblock Continuous maps of the circle without periodic points.
\newblock {\em Israel J. Math.}, 32(4):375--381, 1979.

\bibitem{GOPY}
C.~Grebogi, E.~Ott, S.~Pelikan, and J.~A. Yorke.
\newblock Strange attractors that are not chaotic.
\newblock {\em Phys. D}, 13(1-2):261--268, 1984.

\bibitem{Denjoy}
G.~R. Hall.
\newblock A {$C^{\infty }$} {D}enjoy counterexample.
\newblock {\em Ergodic Theory Dynamical Systems}, 1(3):261--272
(1982), 1981.

\bibitem{Haro}
A.~Haro.
\newblock On strange attractors in a class of pinched skew products.
\newblock preprint.

\bibitem{Keller}
G.~Keller.
\newblock A note on strange nonchaotic attractors.
\newblock {\em Fund. Math.}, 151(2):139--148, 1996.

\bibitem {Ito}
R.~Ito,
\newblock Rotation sets are closed.
\newblock {\em Math.\ Proc.\ Camb.\ Phil.\ Soc.}, 89:107--111,
1981.
\end{thebibliography}
\end{document}